\documentclass[11pt]{article}
\usepackage[margin=1in]{geometry}
\usepackage{amsfonts,amsmath,amssymb}

\PassOptionsToPackage{obeyspaces}{url}
\usepackage[colorlinks=true,citecolor=blue,urlcolor=blue,linkcolor=blue,bookmarksopen=true]{hyperref}
\usepackage{amsthm}
\usepackage{bm}

\usepackage{tikz}
\usetikzlibrary{positioning,shapes,shapes.geometric}

\usepackage{etoolbox}
\patchcmd{\thebibliography}{\leftmargin\labelwidth}{\leftmargin\labelwidth\addtolength\itemsep{-0.1\baselineskip}}{}{}

\usepackage[nameinlink,sort]{cleveref}

\newtheorem{theorem}{Theorem}

\newtheorem{lemma}[theorem]{Lemma}
\crefname{lemma}{lemma}{lemmas}

\newtheorem{corollary}[theorem]{Corollary}
\crefname{corollary}{corollary}{corollaries}

\newtheorem{conj}[theorem]{Conjecture}
\crefname{conj}{conjecture}{conjectures}

\newtheorem{claim}[theorem]{Claim}
\crefname{claim}{claim}{claims}

\newtheorem{prop}[theorem]{Proposition}
\crefname{prop}{proposition}{propositions}

\theoremstyle{definition}

\crefname{defn}{definition}{definitions}

\crefname{remark}{remark}{remarks}

\newtheorem{question}[theorem]{Question}
\crefname{question}{question}{questions}

\crefname{enumi}{part}{parts}

\crefname{equation}{eq.\!}{eqs.\!}

\numberwithin{theorem}{section}

\newcommand*{\eqdef}{\stackrel{\mbox{\normalfont\tiny def}}{=}}   
\newcommand*{\abs}[1]{\lvert{#1}\rvert}                
\newcommand*{\R}{\mathbb{R}}

\title{A discrete view of Gromov's filling area conjecture}

\author{%
    Joseph Briggs\thanks{Department of Mathematics \& Statistics, Auburn University, Auburn, AL. \texttt{\{jgb0059,coc0014\}@auburn.edu}.}
    \and
    Chris Wells\footnotemark[1]
}

\date{}

\begin{document}
\maketitle

\begin{abstract}
    A compact metric surface $M$ \emph{isometrically fills} a closed metric curve $C$ if $\partial M=C$ and $d_M(x,y)=d_C(x,y)$ for every $x,y\in C=\partial M$; that is, $M$ does not introduce any ``shortcuts'' between points on its boundary.
    Gromov's filling area conjecture in differential geometry from 1983 asserts that among all compact, orientable Riemannian surfaces which isometrically fill the Riemannian circle, the one with the smallest surface area is the hemisphere.
    Gromov demonstrated that this is indeed the case if $M$ is homeomorphic to the disk.
    While Gromov's conjecture has since been verified in some other cases, the full conjecture remains unresolved.

    In this paper, we consider a discrete analogue of Gromov's problem, which is likely natural to those who study graph embeddings on arbitrary surfaces.
    Using standard graph-theoretic tools, such as Menger's theorem, we obtain reasonable asymptotic bounds on this discrete variant.
    We then demonstrate how these discrete bounds can be translated to the continuous setting, showing that any isometric filling of the Riemannian circle of length $2\pi$ has surface area at least $1.36\pi$ (the hemisphere has surface area $2\pi$).
    This appears to be the first quantitative lower-bound on Gromov's problem that applies to arbitrary isometric fillings.
\end{abstract}

\section{Gromov's filling area conjecture}

A metric surface $M$ is said to \emph{isometrically fill} a closed, metric curve $C$ if $\partial M=C$ and $d_M(x,y)=d_C(x,y)$ for every $x,y\in C=\partial M$.
In other words, $M$ does not introduce any ``shortcuts'' between points on its boundary.
A natural example arises from when $C=S^1$, the standard circle of circumference $2\pi$, and $M$ is the hemisphere with boundary $S^1$.
It is not hard to convince oneself that $M$ isometrically fills $S^1$.
In contrast, the standard disk $D$ also has $\partial D=S^1$, yet antipodes in $D$ are at distance $2$, as opposed to distance $\pi$, so $D$ does \emph{not} isometrically fill $S^1$.
It seems that the hemisphere is the ``best'' isometric filling of $S^1$.
\begin{conj}[Gromov's filling area conjecture~\cite{gromov_filling}]\label[conj]{gromovConjecture}
    If $M$ is a compact, orientable Riemannian surface which isometrically fills the Riemannian circle of length $2\pi$, then the surface area of $M$ is at least $2\pi$.
\end{conj}
Gromov's conjecture has been verified in the case that $M$ is homeomorphic to a disc~\cite{gromov_filling} and when $M$ has genus 1~\cite{bangert_filling}.

In this paper, we will prove the following quantitative bound related to Gromov's conjecture:
\begin{theorem}\label[theorem]{quantGromov}
    If $M$ is a compact, Riemannian surface which isometrically fills the Riemannian circle of length $2\pi$, then the surface area of $M$ is at least ${\sqrt 3\over 4}\pi^2\approx 1.36\pi$.
\end{theorem}
As far as we are aware, this is the first general lower-bound on the surface area that holds for all considered surfaces.
In fact, it does not even require the assumption of orientability.
\medskip

Our proof of \Cref{quantGromov} follows from bounding a combinatorial relaxation of the problem.
An abstract triangulation is a simplicial complex $K$ of dimension $2$ such that every edge is contained in either $1$ or $2$ triangles.
The boundary $\partial K$ is defined to be the subcomplex induced by those edges contained in exactly one triangle.
For vertices $u,v\in V(K)$, the distance $d_K(u,v)$ is defined to be the graph distance from $u$ to $v$ in the $1$-skeleton of $K$.
For a graph $G$, we say that an abstract triangulation $K$ is an isometric filling of $G$ if $\partial K=G$ and $d_K(x,y)=d_G(x,y)$ for all $x,y\in V(G)$.

\begin{question}
    If $K$ is an isometric filling of $C_n$ (the cycle on $n$ vertices), then how large must $V(K)$ be as $n\to\infty$?
\end{question}
We show that any such triangulation must have at least $n^2/8$ many vertices (see \Cref{discreteGromov}), which is tight up to the constant of $1/8$.
We then use an approximate form of this result to deduce \Cref{quantGromov}.
It is unclear whether or not the constant of $1/8$ can be improved, but we do \emph{not} believe that it can be improved to the point of settling Gromov's conjecture (see \Cref{conclusions}).

\section{Discrete setting}

For an abstract triangulation $K$, we denote the vertices, edges and triangles by $V(K)$, $E(K)$ and $T(K)$, respectively.

We extend the notion of isometric fillings to encapsulate an approximate version.
For a graph $G$ and a number $0<\delta\leq 1$, an abstract triangulation $K$ is said to be a \emph{$\delta$-Lipschitz filling of $G$} if $\partial K=G$ and $d_K(x,y)\geq\delta\cdot d_G(x,y)$ for all $x,y\in V(G)$.
Naturally, $1$-Lipschitz fillings are equivalent to isometric fillings.

\begin{theorem}\label[theorem]{discreteGromov}
    Fix $0<\delta\leq 1$. If $K$ is a $\delta$-Lipschitz filling of $C_n$, then $\abs{V(K)}\geq{\delta^3\over 8}(n-1)^2+{1\over 2}(n-1)$.
\end{theorem}
In the case of $\delta=1$, we thus find that $\abs{V(K)}\geq n^2/8$, as mentioned in the introduction.
The reason for considering $\delta<1$ is that we require an approximate version in order to apply the result to Gromov's conjecture and deduce \Cref{quantGromov}.

We will need two preliminary results in order to establish \Cref{discreteGromov}.

\begin{lemma}\label[lemma]{sperner}
    Let $K$ be an abstract triangulation with $\partial K=C_n$, fix any $x,y\in V(C_n)$ which are not adjacent in $C_n$, and let $L,R$ denote the two connected components of $C_n-x-y$.
    If $S\subseteq V(K)\setminus\{x,y\}$ separates $L,R$ in $K-x-y$, then $S$ contains all internal vertices of a path in the $1$-skeleton of $K$ which connects $x$ and $y$.
\end{lemma}
The proof is very similar to that of Sperner's lemma~{\cite[Chapter 27]{proofs_from_the_book}} and the ``hex theorem''~\cite{gale_hex}.
We remark it has been previously observed that Sperner-type theorems hold for more general classes of manifolds, see in particular \cite{musin2015extensions}.
\begin{proof}
    We begin by claiming that we may suppose that $S$ is disjoint from $V(\partial K)$.
    Indeed, if, say, $\ell\in L\cap S$, then we build a new triangulation $K'$ by introducing a new vertex $\ell'$, which is adjacent to $\ell$ and the two neighbors of $\ell$ in $C_n$, along with the two natural triangles (see \Cref{fig:spernerPadding}).
    We find that $K'$ is also an abstract triangulation with $\partial K'$ being an $n$-cycle, and $x,y$ still lie along this bounding cycle.
    Furthermore, $\partial K'-x-y$ consists of $L'=(L\setminus\{\ell\})\cup\{\ell'\}$ and $R'=R$ and $S$ still separates $L',R'$.
    We may repeat this process until $S$ contains no vertices of the bounding cycle.
    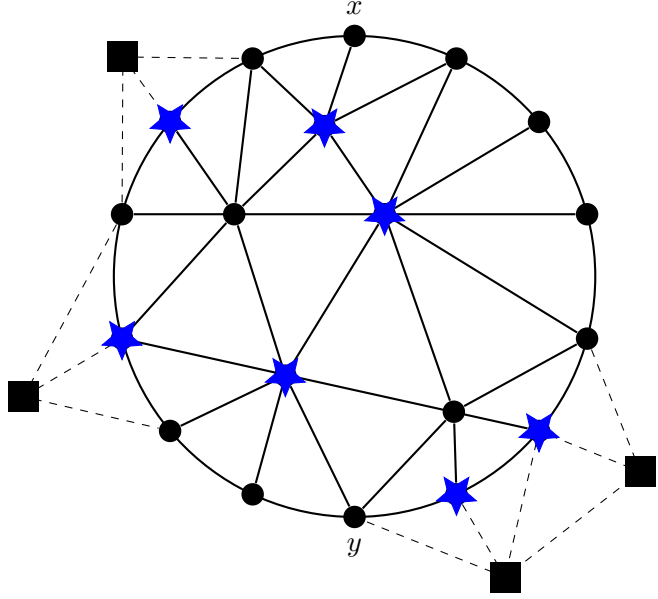
\begin{figure}
        \begin{center}
            \begin{tikzpicture}[scale=0.8]


                \draw[thick, black] (0,0) circle (4);

                \node[fill=black, circle, inner sep=-3, label={$x$}] (yup) at (0,4) {};
                \node[fill=black, circle, inner sep=-3] (ydn) at (0,-4) {};
                \node[below,yshift=-0.4em] (ylabel) at (0,-4) {$y$};

                \node[fill=black, circle, inner sep=-3] (r1) at ({4*cos(deg(65*(pi/180)))},{4*sin(deg(65*(pi/180)))}) {};
                \node[fill=black, circle, inner sep=-3] (r2) at ({4*cos(deg(40*(pi/180)))},{4*sin(deg(40*(pi/180)))}) {};
                \node[fill=black, circle, inner sep=-3] (r3) at ({4*cos(deg(15*(pi/180)))},{4*sin(deg(15*(pi/180)))}) {};

                \node[fill=black, circle, inner sep=-3] (r4) at ({4*cos(deg(65*(pi/180)))},{-4*sin(deg(65*(pi/180)))}) {};
                \node[fill=black, circle, inner sep=-3] (r5) at ({4*cos(deg(40*(pi/180)))},{-4*sin(deg(40*(pi/180)))}) {};
                \node[fill=black, circle, inner sep=-3] (r6) at ({4*cos(deg(15*(pi/180)))},{-4*sin(deg(15*(pi/180)))}) {};

                \node[fill=black, circle, inner sep=-3] (l1) at ({-4*cos(deg(65*(pi/180)))},{4*sin(deg(65*(pi/180)))}) {};
                \node[fill=black, circle, inner sep=-3] (l2) at ({-4*cos(deg(40*(pi/180)))},{4*sin(deg(40*(pi/180)))}) {};
                \node[fill=black, circle, inner sep=-3] (l3) at ({-4*cos(deg(15*(pi/180)))},{4*sin(deg(15*(pi/180)))}) {};

                \node[fill=black, circle, inner sep=-3] (l4) at ({-4*cos(deg(65*(pi/180)))},{-4*sin(deg(65*(pi/180)))}) {};
                \node[fill=black, circle, inner sep=-3] (l5) at ({-4*cos(deg(40*(pi/180)))},{-4*sin(deg(40*(pi/180)))}) {};
                \node[fill=black, circle, inner sep=-3] (l6) at ({-4*cos(deg(15*(pi/180)))},{-4*sin(deg(15*(pi/180)))}) {};

                \node[fill=black, circle, inner sep=-3] (i1) at (-2,{4*sin(deg(15*(pi/180)))}) {};
                \node[fill=black, circle, inner sep=-3] (i2) at (0.5,{4*sin(deg(15*(pi/180)))}) {};
                \node[fill=black, circle, inner sep=-3] (i3) at (-1.15,-1.65) {};
                \node[fill=black, circle, inner sep=-3] (i4) at (1.65,-2.25) {};
                \node[fill=black, circle, inner sep=-3] (i5) at (-0.5,2.5) {};

                \node[fill=black, circle, inner sep=-2.5] (o1) at (-3.85,3.65) {};
                \node[fill=black, circle, inner sep=-2.5] (o2) at (-5.5,-2) {};
                \node[fill=black, circle, inner sep=-2.5] (o3) at (2.5,-5) {};
                \node[fill=black, circle, inner sep=-2.5] (o4) at (4.75,-3.25) {};

                \node[rectangle, minimum size=4mm, fill=black] (or1) at (-3.85,3.65) {};
                \node[rectangle, minimum size=4mm, fill=black] (or2) at (-5.5,-2) {};
                \node[rectangle, minimum size=4mm, fill=black] (or3) at (2.5,-5) {};
                \node[rectangle, minimum size=4mm, fill=black] (or4) at (4.75,-3.25) {};

                \draw[thick] (l3) -- (r3);
                \draw[thick] (l6) -- (r5);
                \draw[thick] (l2) -- (i1);
                \draw[thick] (l1) -- (i1);
                \draw[thick] (l1) -- (i5);
                \draw[thick] (yup) -- (i5);
                \draw[thick] (r1) -- (i5);
                \draw[thick] (i5) -- (i1);
                \draw[thick] (i5) -- (i2);
                \draw[thick] (r1) -- (i2);
                \draw[thick] (r2) -- (i2);

                \draw[thick] (l5) -- (i3);
                \draw[thick] (l4) -- (i3);
                \draw[thick] (ydn) -- (i3);
                \draw[thick] (ydn) -- (i4);
                \draw[thick] (r4) -- (i4);
                \draw[thick] (i4) -- (r6);

                \draw[thick] (i2) -- (r6);
                \draw[thick] (l6) -- (i1);
                \draw[thick] (i1) -- (i3);
                \draw[thick] (i3) -- (i2);
                \draw[thick] (i2) -- (i4);

                \draw[thin, dashed] (l3) -- (-3.85,3.65) -- (l1);
                \draw[thin, dashed] (l6) -- (-5.5,-2);
                \draw[thin, dashed] (l3) -- (-5.5,-2) -- (l5);
                \draw[thin, dashed] (l2) -- (-3.85,3.65);

                \draw[thin, dashed] (r6) -- (4.75, -3.25) -- (2.5,-5) -- (ydn);
                \draw[thin, dashed] (r5) -- (4.75, -3.25);
                \draw[thin, dashed] (r5) -- (2.5,-5);
                \draw[thin, dashed] (r4) -- (2.5,-5);

                \node[star, star points=5, star point ratio=0.45, draw=blue, fill=blue, minimum size=2.5mm] (sr4) at ({4*cos(deg(65*(pi/180)))},{-4*sin(deg(65*(pi/180)))}) {};
                \node[star, star points=5, star point ratio=0.45, draw=blue, fill=blue, minimum size=2.5mm] (sr5) at ({4*cos(deg(40*(pi/180)))},{-4*sin(deg(40*(pi/180)))}) {};

                \node[star, star points=5, star point ratio=0.45, draw=blue, fill=blue, minimum size=2.5mm] (sl2) at ({-4*cos(deg(40*(pi/180)))},{4*sin(deg(40*(pi/180)))}) {};
                \node[star, star points=5, star point ratio=0.45, draw=blue, fill=blue, minimum size=2.5mm] (sl6) at ({-4*cos(deg(15*(pi/180)))},{-4*sin(deg(15*(pi/180)))}) {};

                \node[star, star points=5, star point ratio=0.45, draw=blue, fill=blue, minimum size=2.5mm] (si2) at (0.5,{4*sin(deg(15*(pi/180)))}) {};
                \node[star, star points=5, star point ratio=0.45, draw=blue, fill=blue, minimum size=2.5mm] (si3) at (-1.15,-1.65) {};
                \node[star, star points=5, star point ratio=0.45, draw=blue, fill=blue, minimum size=2.5mm] (si5) at (-0.5,2.5) {};

            \end{tikzpicture}
        \end{center}

        \caption{\label{fig:spernerPadding}Padding $K$ to ensure that $S\cap V(\partial K)=\varnothing$.}
    \end{figure}
    \medskip

    Now, color $v\in V(K)$ \emph{blue} if $v\in S\cup\{x,y\}$, \emph{red} if there is a path in the $1$-skeleton of $K$ from $v$ to $R$ using no vertices of $S\cup\{x,y\}$, and otherwise color $v$ \emph{green}.
    By construction, no red vertex is adjacent to any green vertex.
    Furthermore, since $S$ is disjoint from $\partial K$, the only two blue vertices in $\partial K$ are $x$ and $y$.

    We now build an auxiliary graph $G$ whose vertices are $T(K)\cup\{\partial K\}$ (think of $\partial K$ as the ``exterior face'') where $t_1 t_2\in E(G)$ if $t_1\cap t_2\in E(K)$ and has one red end-point and one blue end-point (see \Cref{fig:sperner} for an example).
    Since no red vertex is adjacent to any green vertex, we observe that each triangle of $T(K)$ is adjacent to either $0$ or $2$ other vertices in $G$.
    Additionally, the only two red--blue edges of $\partial K$ are $xr_1$ and $yr_2$ for some $r_1,r_2\in R$, so $\partial K$ has degree $2$ in $G$.
    Thus, in $G$, the vertex $\partial K$ belongs to a cycle.
    The blue vertices in the edges of $K$ crossed along this cycle in $G$ form a walk from $x$ to $y$, which concludes the proof.
    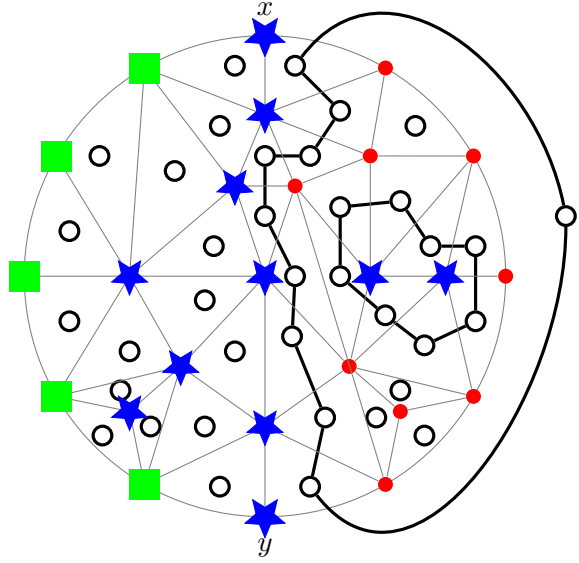
\begin{figure}
        \begin{center}
            \begin{tikzpicture}[scale=0.8,xscale=-1]

                \draw[thin, gray] (0,0) circle (4);

                \node[star, star points=5, star point ratio=0.45, draw=blue, fill=blue, minimum size=2.5mm, label=$x$] (yup) at (0,4) {};
                \node[star, star points=5, star point ratio=0.45, draw=blue, fill=blue, minimum size=2.5mm] (ydn) at (0,-4) {};
                \node[below,yshift=-0.4em] (ylabel) at (0,-4) {$y$};

                \node[fill=red, circle, inner sep=-2] (r1) at ({4*cos(deg(60*(pi/180)))},{4*sin(deg(60*(pi/180)))}) {};
                \node[fill=red, circle, inner sep=-2] (r2) at ({4*cos(deg(30*(pi/180)))},{4*sin(deg(30*(pi/180)))}) {};
                \node[fill=red, circle, inner sep=-2] (r3) at ({4*cos(0)},{4*sin(0)}) {};

                \node[fill=red, circle, inner sep=-2] (r5) at ({4*cos(deg(60*(pi/180)))},{-4*sin(deg(60*(pi/180)))}) {};
                \node[fill=red, circle, inner sep=-2] (r6) at ({4*cos(deg(30*(pi/180)))},{-4*sin(deg(30*(pi/180)))}) {};

                \node[fill=red, circle, inner sep=-2] (l1) at ({-4*cos(deg(60*(pi/180)))},{4*sin(deg(60*(pi/180)))}) {};
                \node[fill=red, circle, inner sep=-2] (l2) at ({-4*cos(deg(30*(pi/180)))},{4*sin(deg(30*(pi/180)))}) {};
                \node[fill=red, circle, inner sep=-2] (l3) at ({-4*cos(0)},{4*sin(0)}) {};

                \node[fill=red, circle, inner sep=-2] (l5) at ({-4*cos(deg(60*(pi/180)))},{-4*sin(deg(60*(pi/180)))}) {};
                \node[fill=red, circle, inner sep=-2] (l6) at ({-4*cos(deg(30*(pi/180)))},{-4*sin(deg(30*(pi/180)))}) {};

                \node[fill=red, circle, inner sep=-2] (i1) at (-1.75,2) {};
                \node[fill=red, circle, inner sep=-2] (i2) at (0.5,1.5) {};
                \node[fill=red, circle, inner sep=-2] (i3) at (-1.4,-1.5) {};
                \node[fill=red, circle, inner sep=-2] (i4) at (1.4,-1.5) {};
                \node[fill=red, circle, inner sep=-2] (i5) at (-0.5,1.5) {};
                \node[fill=red, circle, inner sep=-2] (i6) at (2.25,0) {};
                \node[fill=red, circle, inner sep=-2] (i7) at (-1.75,0) {};
                \node[fill=red, circle, inner sep=-2] (i8) at (-3,0) {};
                \node[fill=red, circle, inner sep=-2] (i9) at (0,2.7) {};
                \node[fill=red, circle, inner sep=-2] (i10) at (0,-2.5) {};
                \node[fill=red, circle, inner sep=-2] (i11) at (-2.25,-2.25) {};
                \node[fill=red, circle, inner sep=-2] (i12) at (2.25,-2.25) {};

                \node[fill=red, circle, inner sep=-2] (c) at (0,0) {};

                \node[circle, draw=black, very thick, inner sep=-2.5] (oc1) at (-0.5,3.5) {};
                \node[circle, draw=black, very thick, inner sep=-2.5] (oc2) at (-2.5,2.5) {};
                \node[circle, draw=black, very thick, inner sep=-2.5] (oc3) at (-1.25,2.75) {};
                \node[circle, draw=black, very thick, inner sep=-2.5] (oc4) at (-0.75,2) {};
                \node[circle, draw=black, very thick, inner sep=-2.5] (oc5) at (0,2) {};
                \node[circle, draw=black, very thick, inner sep=-2.5] (oc6) at (-2.25,1.25) {};
                \node[circle, draw=black, very thick, inner sep=-2.5] (oc7) at (-1.25,1.15) {};
                \node[circle, draw=black, very thick, inner sep=-2.5] (oc8) at (-3.5,0.5) {};
                \node[circle, draw=black, very thick, inner sep=-2.5] (oc9) at (-2.75,0.5) {};
                \node[circle, draw=black, very thick, inner sep=-2.5] (oc10) at (-1.25,0) {};
                \node[circle, draw=black, very thick, inner sep=-2.5] (oc11) at (-2,-0.65) {};
                \node[circle, draw=black, very thick, inner sep=-2.5] (oc12) at (-3.5,-0.75) {};
                \node[circle, draw=black, very thick, inner sep=-2.5] (oc13) at (-2.65,-1.15) {};
                \node[circle, draw=black, very thick, inner sep=-2.5] (oc14) at (-2.25,-1.9) {};
                \node[circle, draw=black, very thick, inner sep=-2.5] (oc15) at (-1.85,-2.35) {};
                \node[circle, draw=black, very thick, inner sep=-2.5] (oc16) at (-1,-2.35) {};
                \node[circle, draw=black, very thick, inner sep=-2.5] (oc17) at (-2.65,-2.65) {};
                \node[circle, draw=black, very thick, inner sep=-2.5] (oc18) at (-0.75,-3.5) {};

                \node[circle, draw=black, very thick, inner sep=-2.5] (ocl1) at (0,1) {};
                \node[circle, draw=black, very thick, inner sep=-2.5] (ocl2) at (-0.5,0) {};
                \node[circle, draw=black, very thick, inner sep=-2.5] (ocl3) at (-0.45,-1) {};
                \node[circle, draw=black, very thick, inner sep=-2.5] (ocl4) at (0.5,-1.25) {};
                \node[circle, draw=black, very thick, inner sep=-2.5] (ocl5) at (1,-0.4) {};
                \node[circle, draw=black, very thick, inner sep=-2.5] (ocl6) at (0.85,0.5) {};

                \node[circle, draw=black, very thick, inner sep=-2.5] (ocl7) at (0.5,3.5) {};
                \node[circle, draw=black, very thick, inner sep=-2.5] (ocl8) at (0.75,2.5) {};
                \node[circle, draw=black, very thick, inner sep=-2.5] (ocl9) at (1.5,1.75) {};
                \node[circle, draw=black, very thick, inner sep=-2.5] (ocl10) at (2.75,2) {};
                \node[circle, draw=black, very thick, inner sep=-2.5] (ocl11) at (3.25,0.75) {};
                \node[circle, draw=black, very thick, inner sep=-2.5] (ocl12) at (3.25,-0.75) {};
                \node[circle, draw=black, very thick, inner sep=-2.5] (ocl13) at (2.25,-1.25) {};
                \node[circle, draw=black, very thick, inner sep=-2.5] (ocl14) at (2.4,-1.9) {};
                \node[circle, draw=black, very thick, inner sep=-2.5] (ocl15) at (2.7,-2.65) {};
                \node[circle, draw=black, very thick, inner sep=-2.5] (ocl16) at (1.9,-2.5) {};
                \node[circle, draw=black, very thick, inner sep=-2.5] (ocl17) at (1,-2.5) {};
                \node[circle, draw=black, very thick, inner sep=-2.5] (ocl18) at (0.75,-3.5) {};
                \draw[very thick] (oc7) -- (oc6) -- (oc9) -- (oc8) -- (oc12) -- (oc13) -- (oc11) -- (oc10) -- (oc7);
                \draw[very thick] (oc1) -- (oc3) -- (oc4) -- (oc5) -- (ocl1) -- (ocl2) -- (ocl3) -- (oc16) -- (oc18);

                \node[circle, draw=black, very thick, inner sep=-2.5] (ocl19) at (-5,1) {};
                \draw[very thick] (oc18) .. controls (-2.5,-5.75) and (-5,-2) .. (ocl19) .. controls (-4.5,3.5) and (-2,5.5) .. (oc1);

                \draw[thin, gray] (yup) -- (i9);
                \draw[thin, gray] (i9) -- (i5);
                \draw[thin, gray] (i5) -- (l5);
                \draw[thin, gray] (i9) -- (i2);
                \draw[thin, gray] (i5) -- (i2);
                \draw[thin, gray] (c) -- (r3);
                \draw[thin, gray] (r1) -- (i1);
                \draw[thin, gray] (r1) -- (i2);
                \draw[thin, gray] (r1) -- (i6);
                \draw[thin, gray] (r2) -- (i6);
                \draw[thin, gray] (i2) -- (i6);
                \draw[thin, gray] (i1) -- (l2);
                \draw[thin, gray] (i1) -- (i5);
                \draw[thin, gray] (i1) -- (i7);
                \draw[thin, gray] (i7) -- (l3);
                \draw[thin, gray] (i7) -- (i5);
                \draw[thin, gray] (i8) -- (l6);
                \draw[thin, gray] (l2) -- (i8);
                \draw[thin, gray] (l2) -- (i7);
                \draw[thin, gray] (l1) -- (i1);
                \draw[thin, gray] (l1) -- (i9);

                \draw[thin, gray] (i5) -- (c);
                \draw[thin, gray] (i2) -- (c);
                \draw[thin, gray] (c) -- (ydn);
                \draw[thin, gray] (c) -- (i4);
                \draw[thin, gray] (i3) -- (c);

                \draw[thin, gray] (i4) -- (r6);
                \draw[thin, gray] (i4) -- (r5);
                \draw[thin, gray] (i4) -- (i6);
                \draw[thin, gray] (i4) -- (i10);
                \draw[thin, gray] (i4) -- (i12);
                \draw[thin, gray] (i12) -- (r6);
                \draw[thin, gray] (r6) -- (i6);
                \draw[thin, gray] (i12) -- (r5);
                \draw[thin, gray] (r5) -- (i10);
                \draw[thin, gray] (i3) -- (i8);
                \draw[thin, gray] (i3) -- (i7);
                \draw[thin, gray] (i3) -- (l6);
                \draw[thin, gray] (i3) -- (i10);
                \draw[thin, gray] (i3) -- (i11);
                \draw[thin, gray] (i11) -- (l6);
                \draw[thin, gray] (i11) -- (l5);
                \draw[thin, gray] (l5) -- (i10);

                \node[rectangle, minimum size=4mm, fill=green] (sq1) at ({4*cos(deg(60*(pi/180)))},{4*sin(deg(60*(pi/180)))}) {};
                \node[rectangle, minimum size=4mm, fill=green] (sq2) at ({4*cos(deg(30*(pi/180)))},{4*sin(deg(30*(pi/180)))}) {};
                \node[rectangle, minimum size=4mm, fill=green] (sq3) at ({4*cos(0)},{4*sin(0)}) {};
                \node[rectangle, minimum size=4mm, fill=green] (sq4) at ({4*cos(deg(60*(pi/180)))},{-4*sin(deg(60*(pi/180)))}) {};
                \node[rectangle, minimum size=4mm, fill=green] (sq5) at ({4*cos(deg(30*(pi/180)))},{-4*sin(deg(30*(pi/180)))}) {};

                \node[star, star points=5, star point ratio=0.45, draw=blue, fill=blue, minimum size=2.25mm] (si1) at (0,2.7) {};
                \node[star, star points=5, star point ratio=0.45, draw=blue, fill=blue, minimum size=2.25mm] (si2) at (0.5,1.5) {};
                \node[star, star points=5, star point ratio=0.45, draw=blue, fill=blue, minimum size=2.25mm] (si3) at (0,0) {};
                \node[star, star points=5, star point ratio=0.45, draw=blue, fill=blue, minimum size=2.25mm] (si4) at (2.25,0) {};
                \node[star, star points=5, star point ratio=0.45, draw=blue, fill=blue, minimum size=2.25mm] (si5) at (-1.75,0) {};
                \node[star, star points=5, star point ratio=0.45, draw=blue, fill=blue, minimum size=2.25mm] (si6) at (-3,0) {};

                \node[star, star points=5, star point ratio=0.45, draw=blue, fill=blue, minimum size=2.25mm] (si8) at (0,-2.5) {};
                \node[star, star points=5, star point ratio=0.45, draw=blue, fill=blue, minimum size=2.25mm] (si10) at (2.25,-2.25) {};
                \node[star, star points=5, star point ratio=0.45, draw=blue, fill=blue, minimum size=2.25mm] (si11) at (1.4,-1.5) {};

            \end{tikzpicture}
        \end{center}
        \caption{\label{fig:sperner}An example of the auxiliary graph $G$.}
    \end{figure}
\end{proof}

We next establish a fairly routine result about the lengths of paths with distinct endpoints which cross some cut in a cycle.
\begin{prop}\label[prop]{pathSum}
    Fix $x,y\in V(C_n)$, which are not adjacent, and let $L,R$ denote the two connected components of $C_n-x-y$.
    For distinct vertices $\ell_1,\dots,\ell_k\in L$ and $r_1,\dots,r_k\in R$,
    \[
        \sum_{i=1}^k d_{C_n}(\ell_i,r_i)\geq {k(k+2)\over 2}.
    \]
\end{prop}
\begin{proof}
    For each $i\in[k]$, let $P_i$ be a shortest path with first vertex $\ell_i$ and last vertex $r_i$, so $\abs{P_i}=d_{C_n}(\ell_i,r_i)+1$.
    Naturally, each $P_i$ contains either $x$ or $y$; without loss of generality, suppose that $P_1,\dots,P_s$ contain $x$ and $P_{s+1},\dots,P_k$ contain $y$.

    Now, consider the path on $2s+1$ many vertices which has $x$ as its center vertex; label the vertices $a_1,a_{2},\dots,a_s,x,b_s,b_{s-1},\dots,b_1$ in order where $a_1,\dots,a_s\in L$ and $b_s,\dots,b_1\in R$.
    Observe that because $\ell_1,\dots,\ell_s$ are distinct elements of $L$, for any $i\in[s]$, it must be the case that $a_i$ belongs to at least $i$ of the paths $P_1,\dots,P_s$.
    Symmetrically, $b_i$ belongs to at least $i$ of the paths $P_1,\dots,P_s$ as well.
    Finally, $x$ belongs to each of these $s$ paths.
    Therefore,
    \[
        \sum_{i=1}^s\abs{P_i}\geq s+2\sum_{i=1}^s i=s^2+2s.
    \]
    By symmetric reasoning about the paths passing through $y$,
    \[
        \sum_{i=s+1}^k\abs{P_i}\geq (k-s)^2+2(k-s).
    \]
    Combining these inequalities and using the convexity of the function $f(t)=t^2$, we bound,
    \[
        \sum_{i=1}^k\abs{P_i}\geq s^2+2s+(k-s)^2+2(k-s)\geq {k^2\over 2}+2k.
    \]
    Therefore,
    \[
        \sum_{i=1}^k d_{C_n}(\ell_i,r_i)=\sum_{i=1}^k\bigl(\abs{P_i}-1\bigr)\geq {k^2\over 2}+2k-k={k(k+2)\over 2}.\qedhere
    \]
\end{proof}

With these two preliminary results in hand, we can now prove our lower-bound on the discrete variant of Gromov's problem.
\begin{proof}[Proof of \Cref{discreteGromov}]
    Fix $x,y\in V(C_n)$ with $d_{C_n}(x,y)=\lfloor n/2\rfloor$.
    Thanks to \Cref{sperner}, we know that if $S\subseteq V(K)\setminus\{x,y\}$ separates $L,R$ in the $1$-skeleton of $K-x-y$, then $S$ contains the interior vertices of some path connecting $x$ and $y$.
    In particular, $\abs S\geq d_K(x,y)-1\geq \delta\cdot d_{C_n}(x,y)-1=\delta\cdot\lfloor n/2\rfloor-1$ since $K$ is a $\delta$-Lipschitz filling of $C_n$.
    Since this is true of any such $S$, we may therefore apply the classical Menger's theorem (see, e.g.,~{\cite[Section 4.2]{west_graphs}}) to conclude that there are at least $\delta\cdot\lfloor n/2\rfloor-1$ many vertex-disjoint paths connecting $L$ and $R$ in the $1$-skeleton of $K-x-y$.
    Label these paths $P_1,\dots,P_k$ and suppose that $P_i$'s endpoints are $\ell_i\in L$ and $r_i\in R$.
    Since the $P_i$'s are vertex-disjoint and we may assume no internal vertex resides in $C_n$, we bound
    \begin{align*}
        \abs{V(K)} &\geq n+\sum_{i=1}^k\bigl(\abs{P_i}-2\bigr)=n-k+\sum_{i=1}^k d_G(\ell_i,r_i)\geq n-k+\delta\sum_{i=1}^k d_{C_n}(\ell_i,r_i)\\
                   &\geq n-k+\delta{k(k+2)\over 2}.
    \end{align*}
    where the final inequality follows from \Cref{pathSum}.
    Now, we can take $\delta\lfloor n/2\rfloor-1\leq k\leq\delta\lfloor n/2\rfloor$, so
    \begin{align*}
        \abs{V(K)} &\geq n-\delta\biggl\lfloor{n\over 2}\biggr\rfloor+{\delta\over 2}\biggl(\delta\biggl\lfloor{n\over 2}\biggr\rfloor-1\biggr)\biggl(\delta\biggl\lfloor{n\over 2}\biggr\rfloor+1\biggr)\\
                   &= n-\delta\biggl\lfloor{n\over 2}\biggr\rfloor+{\delta^3\over 2}\biggl\lfloor{n\over 2}\biggr\rfloor^2-{\delta\over 2}\\
                   &\geq {\delta^3(n-1)^2\over 8}+{n-1\over 2}. \qedhere
    \end{align*}
\end{proof}

While \Cref{discreteGromov} yields a lower-bound on the number of vertices in a $\delta$-Lipschitz filling of $C_n$, we will actually require a lower-bound on the number of triangles in such a filling in order to apply it to Gromov's problem and establish \Cref{quantGromov}.
Luckily, Euler's formula allows us to easily derive such a bound.

\begin{corollary}\label[corollary]{triangleCount}
    Fix $0<\delta\leq 1$ and suppose that $K$ is a $\delta$-Lipschitz filling of $C_n$.
    Let $K'$ denote the triangulation formed by adding $n-2$ triangles to $K$ in order to close the cycle.
    If the Euler characteristic of $K'$ is $\chi$, then $\abs{T(K)}\geq{\delta^3\over 4}(n-1)^2+1-2\chi$.
\end{corollary}
\begin{proof}
    Since $K'$ is a triangulation of a closed surface of Euler characteristic $\chi$, we know that $\abs{V(K')}-\abs{E(K')}+\abs{T(K')}=\chi$.
    Naturally $\abs{E(K')}={3\over 2}\abs{T(K')}$ since every edge of $K'$ resides in two triangles and every triangle has three edges.
    Therefore, $\abs{T(K')}=2\abs{V(K')}-2\chi$.
    Of course, $\abs{V(K')}=\abs{V(K)}$ and $\abs{T(K')}=\abs{T(K)}+n-2$, so, by additionally applying \Cref{discreteGromov},
    \begin{align*}
        \abs{T(K)} &=2\abs{V(K)}-n+2-2\chi\geq 2\biggl({\delta^3(n-1)^2\over 8}+{n-1\over 2}\biggr)-n+2-2\chi\\
                   &={\delta^3(n-1)^2\over 4}+1-2\chi.\qedhere
    \end{align*}
\end{proof}

\section{Continuous setting}
In order to actually apply our discrete inequality to Gromov's problem, we rely on the existence of ``balanced triangulations'' of surfaces.

We define a PL (piecewise linear) metric surface to be a metric surface $(\Omega,\rho)$ along with a \emph{finite} triangulation $T$ of $\Omega$ with the property that for every triangle $\sigma\in T$, there is a function $f_\sigma\colon\sigma\to\R^2$ such that $\rho(x,y)=\lVert f_\sigma(x)-f_\sigma(y)\rVert_2$ for every $x,y\in\sigma$.
Informally, a PL metric surface is formed by gluing together a finite number of Euclidean triangles.

\begin{lemma}\label[lemma]{triangulation}
    Let $M$ be any PL metric surface.
    There exists a sequence $\epsilon=\epsilon(k)$ with $\epsilon\to 0$ as $k\to\infty$ such that there is a triangulation $K=K(k)$ of $M$ with the following properties as $k\to\infty$:
    \begin{enumerate}
        \item Every edge of $K$ has length $\leq\epsilon+o(\epsilon)$,
        \item Every edge of $\partial K$ additionally has length $\geq\epsilon-o(\epsilon)$,
        \item All but $O(1/\epsilon)$ many triangles in $K$ are equilateral with side-length $\epsilon$.
    \end{enumerate}
\end{lemma}
Note that the implicit functions/constants in the little-oh/big-oh notation depend on $M$.
\begin{proof}
    We begin by showing that there is a sequence $\epsilon=\epsilon(k)\to 0$ such that each $e\in E(M)$ can be paritioned into intervals $I_1(e),\dots,I_{s(e)}(e)$, where $I_1(e),\dots,I_{s(e)-1}(e)$ all have length exactly $\epsilon$ and $I_{s(e)}(e)$ has length $\epsilon\pm o(\epsilon)$.

    Set $m=\abs{E(M)}$.\
    Dirichlet's approximation theorem~{\cite[Theorem 1B in Chapter II]{lang_diophantine}} tells us that for any positive integer $k$, we can find integers $\{p_k(e):e\in E(M)\}$ and an integer $q_k$ with $1\leq q_k\leq k$ satisfying
    \[
        \biggl\lvert\ell(e)-{p_k(e)\over q_k}\biggr\rvert\leq{1\over q_k k^{1/m}}\qquad\text{for all }e\in E(M),
    \]
    where $\ell(e)$ denotes the length of the edge $e$.
    Set
    \[
        L(k)\eqdef\biggl\lceil{k^{1/m}\over\log k}\biggr\rceil,\qquad\text{and}\qquad\epsilon(k)\eqdef{1\over q_k L(k)}.
    \]
    Since $m=\abs{E(M)}$ is fixed, we find that $\epsilon\to 0$ as $k\to\infty$ as needed.
    By defining $s(e)=L(k)\cdot p_k(e)$, we can then partition $e$ into intervals $I_1(e),\dots,I_{s(e)}(e)$ such that $I_1(e),\dots,I_{s(e)-1}(e)$ each have length $\epsilon$ and $I_{s(e)}(e)$ satisfies
    \begin{align*}
        \bigl\lvert \ell\bigl(I_{s(e)}(e)\bigr)-\epsilon\bigr\rvert &= \biggl\lvert \biggl(\ell(e)-\sum_{i=1}^{s(e)-1}\ell\bigl(I_i(e)\bigr)\biggr) - \epsilon\biggr\rvert = \bigl\lvert \ell(e)-s(e)\epsilon\bigr\rvert \\
                                                                    &= \biggl\lvert \ell(e)-{p_k(e)\over q_k}\biggr\rvert\leq {1\over q_k k^{1/m}}=O\biggl({\epsilon\over\log k}\biggr)=o(\epsilon),
    \end{align*}
    as needed.
    \medskip

    With this step done, the fact that $M$ consists of finitely many triangles means that it suffices to prove the following in order to finish the proof of the claim:
    \begin{claim}
        Fix a triangle $T$ in $\R^2$.
        For a sequence $\epsilon=\epsilon(k)\to 0$, consider any subdivision of the edges of $T$ into intervals of lengths $\epsilon\pm o(\epsilon)$.
        There is a triangulation of $T$ that does not introduce any new points along $\partial T$ such that every edge has length at most $\epsilon+o(\epsilon)$ and all but at most $O(1/\epsilon)$ many triangles are equilateral with side-length $\epsilon$.
    \end{claim}
    Here, the little-oh/big-oh may depend on the triangle $T$.
    \begin{proof}
        Consider the natural triangulation of $\R^2$ by equilateral triangles of side-length $\epsilon$ (i.e.\ the dual of the standard hexagonal lattice).
        Let $P$ denote those triangles which are contained completely within the interior of the triangle $T$ (see \Cref{fig:equilateral}).
        By moving $P$ by a small isometry, we may suppose that every point in $V(\partial P)$ has a unique closest point in $V(\partial T)$ and vice versa.

        \begin{figure}
            \begin{center}
                \begin{tikzpicture}[x={(1,0)},y={(0.5,0.866)}]
                    \draw[ultra thick] (-0.5,-0.5)--(1,5.5)--(6.8,-1.5)--cycle;
                    \draw (4,-1)--(6,-1);
                    \draw (0,0)--(5,0);
                    \draw (0,1)--(4,1);
                    \draw (1,2)--(3,2);
                    \draw (1,3)--(3,3);
                    \draw (1,4)--(2,4);
                    \draw (1,5)--(3,3);
                    \draw (1,4)--(6,-1);
                    \draw (1,3)--(5,-1);
                    \draw (1,2)--(4,-1);
                    \draw (1,1)--(2,0);
                    \draw (0,1)--(1,0);
                    \draw (0,0)--(0,1);
                    \draw (1,0)--(1,5);
                    \draw (2,0)--(2,4);
                    \draw (3,0)--(3,3);
                    \draw (4,-1)--(4,1);
                    \draw (5,-1)--(5,0);
                \end{tikzpicture}
            \end{center}
            \caption{\label{fig:equilateral}A near-filling of $T$ by $\epsilon$-equilateral triangles.}
        \end{figure}
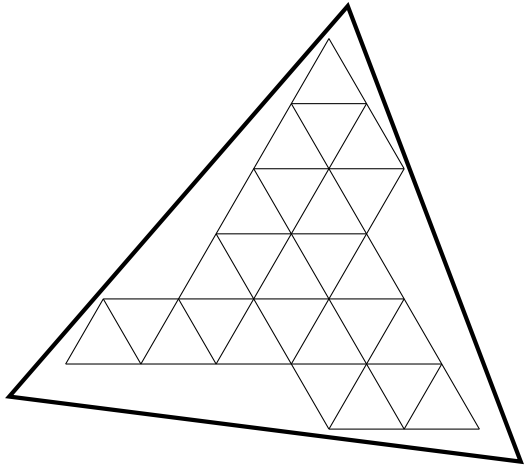

        Consider the annulus $A=T\setminus P$, so $\partial A=\partial T\sqcup\partial P$.
        We will show that we can triangulate $A$ without introducing any new points along $\partial A$ using at most $O(1/\epsilon)$ many triangles wherein every edge has length at most $\epsilon+o(\epsilon)$.
        This will conclude the proof of the claim since $P$ consists only of equilateral triangles with side-length $\epsilon$.
        \medskip

        For an edge $e\in E(\partial P)$, let $t(e)$ denote the equilateral triangle using $e$ as one of its sides that is \emph{not} contained in $P$.
        By the definition of $P$, we know that $t(e)$ must intersect $\partial T$.
        Since the edges of $E(\partial T)$ all have length $\leq\epsilon+o(\epsilon)$, this implies that any point in $V(\partial P)$ is at distance at most $2\epsilon+o(\epsilon)$ from some point in $V(\partial T)$.

        Next, consider a point $v\in V(\partial T)$.
        If $v$ belongs to the triangle $t(e)$ for some $e\in E(\partial P)$, then $v$ is at distance at most $\epsilon$ from some point on $V(\partial P)$.
        Otherwise, let $y\in V(\partial P)$ be the closest point to $v$ and consider the triangles $t(xy)$ and $t(yz)$ where $x,z$ are the neighbors of $y$ in $\partial P$.
        Observe that $y$ belongs to either $1$, $2$ or $3$ triangles since $v$ is in neither $t(xy)$ nor $t(yz)$.
        If $y$ belongs to $3$ triangles, then $x,y,z$ are colinear and so $v$ is certainly at distance at most $2\epsilon+o(\epsilon)$ from $y$.
        Otherwise, $y$ belongs to either $1$ or $2$ triangles, in which case the distance from $v$ to $y$ is bounded above $C\epsilon$ where $C$ is some constant depending only on the angle at the vertex of $T$ between the triangles $t(xy),t(yz)$.
        \medskip

        Connect each point of $V(\partial P)$ to its closest neighbor in $V(\partial T)$ and connect each point in $V(\partial T)$ to its closest neighbor in $V(\partial P)$.
        From the previous paragraph, each of these edges has length at most $C\epsilon$.
        Furthermore, due to the triangle inequality, no two edges cross and so we have partitioned $A$ into triangles and quadrilaterals.
        By connecting some pair of opposite corners of each quadrilateral, we arrive at a triangulation of $A$ (with no new points in $\partial A$) such that every edge has length at most $2C\epsilon$; call this triangulation $X$.
        Naturally, there are at most $2\abs{V(\partial T)}=O(1/\epsilon)$ many triangles in $X$.

        We now perform ``partial barycentric subdivisions'' on all triangles in $X$ to form a triangulation $X'$.
        Here, a ``partial barycentric subdivision'' is a barycentric subdivision, except we \emph{do not} bisect any edge that belongs to $E(\partial A)$ (see \Cref{fig:barycentric}).
        In doing so, the length of the longest edge in $X'$ is at most the maximum of $\epsilon+o(\epsilon)$ and $2/3$ the length of the longest edge in $X$.
        Additionally, $X'$ contains at most $5$ times the number of triangles as does $X$.
        Thus, by repeating this procedure at most $\lceil\log_{3/2}(2C)\rceil$ many times, we arrive at a triangulation of $A$ with no new points on $\partial A$ which has every edge of length $\leq\epsilon+o(\epsilon)$ and has at most $5^{\lceil\log_{3/2}(2C)\rceil}\cdot O(1/\epsilon)=O(1/\epsilon)$ many triangles, as needed.
        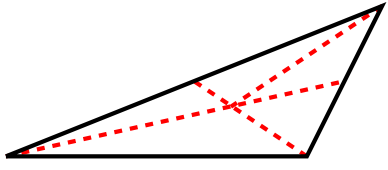
\begin{figure}[ht]
            \begin{center}
                \begin{tikzpicture}
                \begin{scope}
        [cm={1,0,1.5,1,(0,0)}] 
                    \draw[red,ultra thick, dashed] (0,0)--(3,1);
                    \draw[red,ultra thick, dashed] (4,0)--(1,1);
                    \draw[red,ultra thick, dashed] (2,2)--(2,0.66);
                    \draw[ultra thick] (0,0)--(4,0)--(2,2)--(0,0);
                    \end{scope}
                \end{tikzpicture}
            \end{center}
            \caption{\label{fig:barycentric}An example of partial barycentric subdivision where the bottom edge of the triangle is not subdivided.}
        \end{figure}
    \end{proof}
\end{proof}

Now that we know that PL metric surfaces can be triangulated by mostly equilateral triangles, we can apply our discrete bound to Gromov's problem on such surfaces.

\begin{theorem}\label[theorem]{gromovForPL}
    Fix $0<\delta\leq 1$.
    If $M$ is any PL metric surface which is a $\delta$-Lipschitz filling of the Riemannian circle of circumference $\ell$, then the surface area of $M$ is at least ${\sqrt{3}\over 16}\delta^3\ell^2$.
\end{theorem}
\begin{proof}
    Fix a sequence $\epsilon=\epsilon(k)\to 0$ and triangulations $K=K(k)$ of $M$ as in \Cref{triangulation}.
    By design, every edge of $\partial K$ has length $\epsilon\pm o(\epsilon)$ and so $\partial K=C_n$ where $n=\bigl(1\pm o(1)\bigr){\ell\over\epsilon}$.
    Furthermore,
    for any vertices $x$ and $y$ of $K$,
    \[
        d_{\partial K}(x,y)=\bigl(1\pm o(1)\bigr){d_{\partial M}(x,y)\over \epsilon}.
    \]
    Here, we are treating $K$ as an abstract triangulation and so $d_{\partial K}$ refers to graph-distance.
    Furthermore, since every edge of $K$ has length $\leq\epsilon+o(\epsilon)$, we find that
    \[
        d_K(x,y)\geq\bigl(1-o(1)\bigr){d_M(x,y)\over\epsilon}.
    \]
    Since $M$ is a $\delta$-Lipschitz filling of $\partial K$, we then have
    \[
        d_K(x,y)\geq\bigl(1-o(1)\bigr){d_M(x,y)\over\epsilon}\geq\bigl(1-o(1)\bigr){\delta\cdot d_{\partial M}(x,y)\over\epsilon}\geq\bigl(1- o(1)\bigr)\cdot\delta\cdot d_{\partial K}(x,y).
    \]
    That is, $K$ is a $\bigl(1-o(1)\bigr)\delta$-Lipschitz filling of $C_n$.
    Now, letting $\chi$ denote the Euler characteristic of the surface $M$ after filling in the missing disk, \Cref{triangleCount} allows us to bound
    \[
        \abs{T(K)}\geq \bigl(1-o(1)\bigr){\delta^3\over 4}(n-1)^2+1-2\chi=\bigl(1-o(1)\bigr){\delta^3\over 4} \cdot{\ell^2\over\epsilon^2},
    \]
    since $\chi$ depends only on $M$ and not on $k$.
    Now, all but $O(1/\epsilon)$ of the triangles in $T(K)$ are equilateral with side-length $\epsilon$ and so the surface area of $M$ is at least
    \[
        \biggl(\abs{T(K)}-O\biggl({1\over\epsilon}\biggr)\biggr)\cdot{\sqrt{3}\over 4}\epsilon^2\geq\bigl(1-o(1)\bigr){\sqrt{3}\over 16}\delta^3\ell^2-O(\epsilon).
    \]
    Taking the limit as $k\to\infty$ (and thus $\epsilon\to 0$) then yields the claim.
\end{proof}

\begin{corollary}
    Fix $0<\delta\leq 1$.
    If $M$ is any compact Riemannian surface which is a $\delta$-Lipschitz filling of the Riemannian circle of circumference $\ell$, then the surface area of $M$ is at least ${\sqrt{3}\over 16}\delta^3\ell^2$.
\end{corollary}
\begin{proof}
    Fix $\epsilon>0$.
    It is known\footnote{See, e.g.,~{\cite[Section 1]{cheeger_flat}}.} that one can build a PL metric surface $P_\epsilon$ such that there is a homeomorphism $f\colon M\to P_\epsilon$ satisfying
    \[
        (1-\epsilon)d_M(x,y)\leq d_P\bigl(f(x),f(y)\bigr)\leq(1+\epsilon)d_M(x,y)\qquad\text{for all }x,y\in M.
    \]
    Naturally, this means that the surface area of $M$ is at least $(1-\epsilon)^2$ times the surface area of $P_\epsilon$.
    Furthermore $P_\epsilon$ is a $(1-\epsilon)\delta$-filling of $\partial P_\epsilon$, which is the Riemannian circle of circumference at least $(1-2\epsilon)\ell$.
    We therefore apply \Cref{gromovForPL} to $P_\epsilon$ and take the limit as $\epsilon\to 0$ in order to establish the claim.
\end{proof}

\section{Conclusions}\label{conclusions}

It is incredibly unfortunate that we have been unable to resolve whether or not our discrete variant of Gromov's problem is actually equivalent to Gromov's problem.
We believe this to \emph{not} be the case, but have so-far been unable to find a construction separating the two.
Letting $D(n;\epsilon)$ denote the minimum number of vertices in a $(1-\epsilon)$-Lipschitz filling of $C_n$, define the limit
\[
    D^*\eqdef\liminf_{\epsilon\to 0^+}\liminf_{n\to\infty}{D(n;\epsilon)\over n^2}.
\]
As demonstrated in this paper, any compact Riemannian surface which isometrically fills the Riemannian circle of circumference $2\pi$ must have surface area at least $2\sqrt 3 \pi^2\cdot D^*$.
Since we know that there are isometric fillings of surface area $2\pi$ (namely, the hemisphere), we bound
\[
    D^*\leq {1\over\pi\sqrt{3}}\approx 0.1838.
\]
Without relying on balanced triangulations of the hemisphere, however, our best explicit constructions have shown only that
\[
    D^*\leq {3\over 16}=0.1875.
\]
That is to say, we have not devised any scheme by which to build interesting examples for our discrete problem that do not somehow invoke the continuous analogue.
Indeed, based solely on the work in this paper, it could still be possible that $D^*={1\over\pi\sqrt{3}}$ (which would verify Gromov's conjecture).
However, we believe strongly that this is not the case:

\begin{conj}
    $\displaystyle D^*<{1\over\pi\sqrt 3}$.
\end{conj}
That is to say, we believe that our discrete problem is a proper relaxation of Gromov's original problem and cannot be used to settle Gromov's conjecture.
\medskip

Next, recall that we proved in \Cref{discreteGromov} that
\[
    D^*\geq{1\over 8}=0.125.
\]
At this time, we have no evidence to suggest nor deny that this is the correct answer.
In fact, the authors have repeatedly changed their beliefs in this regard!
So, instead of conjecture a truth, we simply pose the question:
\begin{question}
    Is $\displaystyle D^*>{1\over 8}$?
\end{question}
We quickly remark that we would be happy with solely determining whether or not $\liminf_{n\to\infty}{D(n;0)\over n^2}>{1\over 8}$.
It seems reasonable to suspect that $\liminf_{n\to\infty}{D(n;\epsilon)\over n^2}$ is a continuous function of $\epsilon$ for $\epsilon\in[0,1]$,\footnote{And also that $\lim_{n\to\infty}{D(n;\epsilon)\over n^2}$ exists for every $\epsilon\in[0,1]$.} but we do not see an immediately obvious proof of this fact.

\bibliographystyle{abbrv}
\bibliography{references}

\end{document}